\documentclass[10pt,a4paper]{article}
\usepackage[latin1]{inputenc}
\usepackage{amsmath}
\usepackage{amsfonts}
\usepackage{amssymb}
\usepackage{graphicx}
\usepackage{amsthm} 	
\usepackage{tikz}
\usepackage{hyperref}
\usepackage{todonotes}
\usetikzlibrary{backgrounds}	
\usepackage[toc]{appendix}

\usepackage{graphicx,xcolor}	

\usepackage{xcolor}

\usepackage{esint}				

\usepackage[a4paper,top=2.1cm,bottom=2.60cm,left=3.8cm,right=3.8cm]{geometry} 

\allowdisplaybreaks

\newtheorem{theorem}{Theorem}[section]
\newtheorem{proposition}[theorem]{Proposition}
\newtheorem{corollary}[theorem]{Corollary}

\newtheorem{remark}[theorem]{Remark}
\newtheorem{lemma}[theorem]{Lemma}

\title{Singular behavior for a multi-parameter periodic Dirichlet problem}
\author{Matteo Dalla Riva\thanks{Dipartimento di Ingegneria, Universit\`a degli Studi di Palermo, Viale delle Scienze, Ed.~8, 90128 Palermo, Italy.} ,  Paolo Luzzini\thanks{Dipartimento di Matematica ``Tullio Levi Civita,'' Universit\`a degli Studi di Padova, Via Trieste 63, 35121 Padova, Italy.} , Paolo Musolino\thanks{Dipartimento di Scienze Molecolari e Nanosistemi, Universit\`a Ca' Foscari Venezia, Via Torino 155, 30172 Venezia Mestre, Italy}}

\date{November 21, 2022}

\begin{document}

\maketitle

\noindent
{\bf Abstract:}  We consider a Dirichlet problem for the Poisson equation in a periodically perforated domain. The geometry of the domain is controlled by two parameters: a real number $\epsilon>0$ proportional to the radius of the holes  and a map $\phi$, which models the shape of the holes. So, if  $g$ denotes the Dirichlet boundary datum and $f$ the Poisson datum, we have a solution for each quadruple $(\epsilon,\phi,g,f)$. Our aim is to study how the solution depends on $(\epsilon,\phi,g,f)$, especially when $\epsilon$ is very small  and the holes narrow to points. In contrast with previous works, we don't introduce the assumption that $f$ has zero integral on the fundamental periodicity cell. This brings in a certain singular behavior  for   $\epsilon$ close to $0$. We show that, when the dimension $n$ of the ambient space is greater than or equal to $3$, a suitable restriction of the solution can be represented with an analytic map of the quadruple $(\epsilon,\phi,g,f)$ multiplied by the factor $1/\epsilon^{n-2}$. In case of dimension $n=2$, we have to add $\log \epsilon$ times the integral of $f/2\pi$.
 \vspace{9pt}

\vspace{9pt}


\section{Introduction}\label{sec:intro}

The object of this paper is to study a singular perturbation of a Dirichlet-Poisson problem in a periodically perforated domain. The aim is to  show that the solution can be written as a combination of real analytic maps and--possibly singular but completely known--elementary functions of the perturbation parameters. The geometry of the problem is controlled by two parameters: a positive real number $\epsilon$ that determines the size of the holes and  a shape function $\phi$  that deforms the boundary of a certain reference domain $\Omega$ into the shape of the holes. To wit, the holes are shifted copies of $\epsilon\phi(\partial\Omega)$ and thus, for $\epsilon$ that tends to zero, they shrink down to points. Next, a function $g$ denotes the Dirichlet datum and a function $f$ the Poisson datum and the problem is such that we have a unique solution for each choice of the four variables  $\epsilon$, $\phi$, $g$, and $f$. So it makes sense to ask ourselves what we can say of the map that takes a quadruple $(\epsilon,\phi,g,f)$ to the corresponding solution. In particular, we want to see what happens  when $\epsilon$ is close to zero and the holes are shrinking to points.  

The interest for periodic problems for the Laplace equations is (in part) motivated by their relevance in the applications. For instance, problems of this kind  arise in the study of  effective properties of composite materials. The reader may find some examples in the works of  Ammari, Kang, and Lim \cite{AmKaLi06}, Ammari, Kang, and Touibi \cite{AmKaTo05},  Dryga\'s, Gluzman, Mityushev, and Nawalaniec \cite{DrGlMiNa20}, Gluzman, Mityushev, and Nawalaniec \cite{GlMiNa18}, and Kapanadze, Mishuris, and Pesetskaya \cite{KaMiPe15, KaMiPe15bis}.

Indeed, we have ourselves already written on this topic. In particular, the problem of this paper is very similar to those of \cite{Mu12} and \cite{Mu13}. There is a critical difference though: in \cite{Mu12} we take the Poisson datum $f$ equal to $0$ and in \cite{Mu13} the function $f$ is required to have zero integral on the periodicity cell, whereas here we abandon this assumption. As a consequence, we have to deal with a specific singular behavior that appears for $\epsilon$   close to zero: If the ambient space has dimension $n\ge 3$, the solution shows a singularity of the order of $1/\epsilon^{n-2}$, and, for $n=2$,  a $(\log\epsilon)$-singularity. Also, in  the previous works \cite{Mu12,Mu13} the number $\epsilon$ was solely responsible for the geometric deformation  of the problem, and thus only homothetic transformations of the holes were allowed. Here, instead,  the holes can change shape according to the function $\phi$  and  we  analyze the joint dependence on the set of variables $(\epsilon,\phi,g,f)$.

We now describe our problem in detail.   We fix once for all 
\begin{equation*}
n \in \mathbb{N} \setminus\{0,1\} \qquad\text{and}\qquad  \,\, q_{11},\ldots,q_{nn}\in \mathopen]0,+\infty[\,,
\end{equation*}
where $\mathbb{N}$ denotes the set of natural numbers including 0. We set
 \[
q :=
 \begin{pmatrix}
  q_{11} & 0 & \cdots & 0 \\
  0 & q_{22} & \cdots & 0 \\
  \vdots  & \vdots  & \ddots & \vdots  \\
  0 & 0 & \cdots & q_{nn} 
 \end{pmatrix}\qquad\text{and} \qquad Q := \prod_{j=1}^n \mathopen]0,q_{jj}\mathclose[ \subseteq \mathbb{R}^n.
 \]
The set $Q$ is the fundamental periodicity cell and the diagonal matrix $q$ is the periodicity matrix associated 
with the cell $Q$. 

  We consider a set $\Omega \subseteq \mathbb{R}^n$ satisfying the following assumption:
\begin{equation}\label{Omega_def}\begin{split}& \Omega \mbox{ is}\mbox{ a bounded open connected subset of } \mathbb{R}^n \mbox{ of class } C^{1,\alpha} \\
&\mbox {such that }\mathbb{R}^n \setminus \overline{\Omega} \mbox{ is connected.}
\end{split}\end{equation}
In \eqref{Omega_def}, as well as in the rest of the paper,   $\alpha$ is a fixed number in the open interval $\mathopen]0,1[$ and the symbol $\overline{\cdot}$ denotes the closure of a set.  For the definition of sets and functions of the Schauder class $C^{j,\alpha}$ ($j \in \mathbb{N}$)  we refer, e.g., to Gilbarg and
Trudinger~\cite{GiTr83}.

The boundary $\partial\Omega$ of $\Omega$  plays the role of a reference set and the boundary of the holes is obtained rescaling and shifting the image of $\partial\Omega$ under a suitable map $\phi$. The set of the functions  $\phi$ that we allow is denoted by  $\mathcal{A}_{\partial \Omega}^{1,\alpha}$  and consists of the functions of  $C^{1,\alpha}(\partial\Omega, \mathbb{R}^{n}):=(C^{1,\alpha}(\partial\Omega))^n$  that are injective and have injective differential at all points  of $\partial\Omega$. 
We can verify that $\mathcal{A}_{\partial \Omega}^{1,\alpha}$ is open 
in $ C^{1,\alpha}(\partial\Omega, \mathbb{R}^{n})$
(see, {e.g.}, Lanza de Cristoforis and Rossi \cite[Lem.~2.2, p.~197]{LaRo08}  
and \cite[Lem.~2.5, p.~143]{LaRo04}). Moreover, if $\phi \in\mathcal{A}_{\partial \Omega}^{1,\alpha}$, then the Jordan-Leray separation theorem (see, {e.g}, Deimling \cite[Thm.~5.2, p.~26]{De85}) ensures that  $\phi(\partial \Omega)$ splits 
$\mathbb{R}^n$ into exactly two open connected components. We denote  by
$\mathbb{I}[\phi]$ the bounded one. 

 Next we fix a point
 \[
 p \in Q\, ,
 \] 
 which is the point where the hole in the  reference periodicity cell shrinks to.
It will be convenient to consider perturbations around a fixed 
\[
\phi_0 \in \mathcal{A}_{\partial\Omega}^{1,\alpha}\,.
\] 
There is no loss of generality in this choice, because $\phi_0$ can be any function of  $\mathcal{A}_{\partial\Omega}^{1,\alpha}$. The advantage is that there exist a real number $\epsilon_0>0$ and an open neighborhood $\mathcal{O}_{\phi_0}$ of $\phi_0$ in $\mathcal{A}_{\partial\Omega}^{1,\alpha}$   such that 
\begin{equation}\label{e0}
p+\epsilon \overline{\mathbb{I}[\phi]} \subseteq Q \ \ \ \ \forall (\epsilon,\phi) \in \mathopen]-\epsilon_0,\epsilon_0\mathclose[\times \mathcal{O}_{\phi_0}.\end{equation}
Then, for these $\epsilon$'s and $\phi$'s we can define the hole 
\[
\Omega_{\epsilon,\phi} := p+\epsilon {\mathbb{I}[\phi]} \ \ \ \ \forall (\epsilon,\phi) \in \mathopen]-\epsilon_0,\epsilon_0\mathclose[\times \mathcal{O}_{\phi_0}\,, 
\]
which is contained in $Q$, has size proportional to $\epsilon$, and shape determined by $\phi$. When $\epsilon$ tends to $0$, the hole shrinks toward the point $p$ while its shape changes according to $\phi$ (see Figure \ref{fig:cell}).

 \begin{figure}
\begin{center}
\includegraphics[width=4.4in]{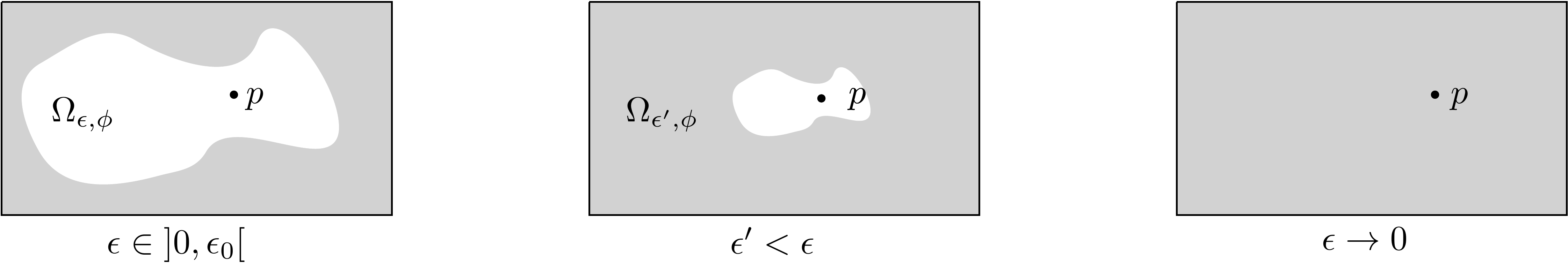}
\caption{{ A $2$-dimensional  example  of  perforated reference periodicity cell. The hole $\Omega_{\epsilon,\phi}$ shrinks toward the point $p$ when $\epsilon$ tends to $0$. }}
\label{fig:cell}
\end{center}
\end{figure}

The periodic set of holes is given by 
\[
\mathbb{S}[\Omega_{\epsilon,\phi}] := \bigcup_{z\in \mathbb{Z}^n}(qz+\Omega_{\epsilon,\phi})\,,
\]
and the periodic domain where we define the Poisson equation is 
\begin{equation*}
\mathbb{S}[\Omega_{\epsilon,\phi}]^-:= \mathbb{R}^n \setminus \overline{\mathbb{S}[\Omega_{\epsilon,\phi}]} \qquad \forall (\epsilon,\phi) \in \mathopen]-\epsilon_0,\epsilon_0\mathclose[\times \mathcal{O}_{\phi_0}\,,
\end{equation*}
that is, the domain obtained removing from $\mathbb{R}^n$ the periodic set of  holes 
$\overline{\mathbb{S}[\Omega_{\epsilon,\phi}]}$  (see Figure \ref{fig:geom-sett}). When $\epsilon$ 
 approaches zero, the hole in the cell $qz+Q$ shrinks toward $qz+p$.

\begin{figure}
\begin{center}
\includegraphics[width=3.6in]{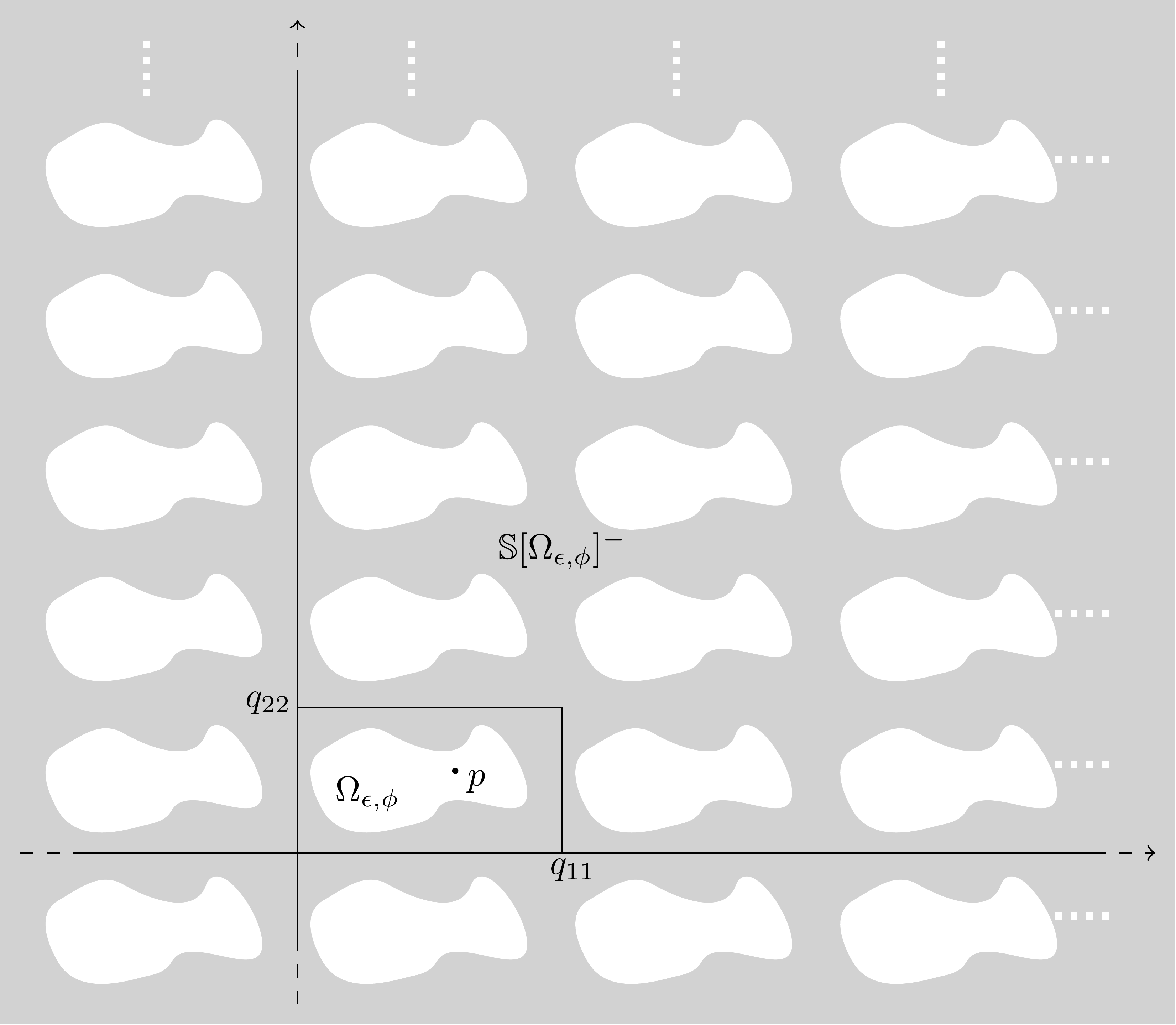}
\caption{{ A $2$-dimensional  example  of   periodically perforeted set
$\mathbb{S}[\Omega_{\epsilon,\phi}]^-$. }}
\label{fig:geom-sett}
\end{center}
\end{figure}

We now introduce suitable spaces for the functional data of the problem. For the right-hand side of the  Dirichlet boundary condition we take a function 
 \[
 g\in C^{1,\alpha}(\partial\Omega)\,,
 \]
 which we properly transplant to be defined on $\partial\Omega_{\epsilon,\phi}=p+\epsilon\phi(\partial\Omega)$.  As for the Poisson datum, regularity has to be chosen more carefully.  Lanza de Cristoforis in \cite{La05}   and Preciso in \cite{Pr98, Pr00} shown that  Roumieu   analytic functions   produce analytic composition operators and analytic Newtonian potentials, a feature that will come handy later on in our analysis. Moreover, since we are dealing with a periodic problem, we have to take periodicity into account. So,   the right-hand side of the Poisson equation will be a  function 
 \[
 f\in C^0_{q,\omega,\rho}(\mathbb{R}^n)\,,
 \] 
 where $\rho>0$ is fixed and $C^0_{q,\omega,\rho}(\mathbb{R}^n)$ denotes the Roumieu class of $q$-periodic analytic functions (see \eqref{proum} for the exact  definition of $C^0_{q,\omega,\rho}(\mathbb{R}^n)$, see also \cite{Mu13}).

 All the ingredients are now introduced and we can state our Dirichlet-Poisson problem. For a quadruple $(\epsilon,\phi,g,f) \in \mathopen]0,\epsilon_0\mathclose[\times \mathcal{O}_{\phi_0} \times C^{1,\alpha}(\partial \Omega) \times C^0_{q,\omega,\rho}(\mathbb{R}^n)$, we look at  
\begin{equation}
\label{bvpe}
\left\{
\begin{array}{ll}
\Delta u (x)=f(x) & \forall x\in \mathbb{S}[\Omega_{\epsilon,\phi}]^-,\,
\\
u(x+ q z)=u(x) &\forall x \in \overline{\mathbb{S}[\Omega_{\epsilon,\phi}]^-},\, \forall z \in \mathbb{Z}^n,
\\
u(x)= g\circ \phi^{(-1)}(\epsilon^{-1}(x- p))
& \forall x\in \partial\Omega_{\epsilon,\phi}\,.
\end{array}
\right.
\end{equation}
It is well known that problem \eqref{bvpe} has a unique solution and that such solution belongs to $C^{1,\alpha}(\overline{\mathbb{S}[\Omega_{\epsilon,\phi}]^-} )$ (see Theorem \ref{biju}, see also \cite[Prop.~2.2]{Mu13}). To emphasize the dependence  on $(\epsilon,\phi,g,f)$, we denote it by  
\[
u[\epsilon,\phi,g,f]\,.
\]

Then our goal is to describe the map 
\[
\mathopen]0,\epsilon_0\mathclose[\times \mathcal{O}_{\phi_0} \times C^{1,\alpha}(\partial \Omega) \times C^0_{q,\omega,\rho}(\mathbb{R}^n)\ni (\epsilon,\phi,g,f)\mapsto u[\epsilon,\phi,g,f]\in C^{1,\alpha}(\overline{\mathbb{S}[\Omega_{\epsilon,\phi}]^-} )\,.
\]
We observe, however, that the space in the right-hand side depends on $\epsilon$, and so it is not suited to be the codomain of a function that depends on $\epsilon$ itself. To fix this inconvenience, we take   an open bounded set $V$ compactly contained in $\mathbb{R}^n \setminus (p+q\mathbb{Z}^n)$ and 
for $\epsilon$ sufficiently small 
consider the    restriction of $u[\epsilon,\phi,g,f]$ to $\overline{V}$.  So that $u[\epsilon,\phi,g,f]_{|\overline{V}}$ belongs to $C^2(\overline{V})$, a space that does not depend on $\epsilon$. 

In our main Theorem \ref{thm:repsol} we see that,  for a possibly smaller $\epsilon_0>0$ and possibly shrinking the open neighborhood $\mathcal{O}_{\phi_0}$  of $\phi_0$, the map that takes $(\epsilon,\phi,g,f)$ to  $u[\epsilon,\phi,g,f]_{|\overline{V}}$ is described by the formula
\[
u[\epsilon,\phi,g,f]_{|\overline{V}}=\frac{1}{\epsilon^{n-2}}\mathfrak{U}[\epsilon,\phi,g,f]+\delta_{2,n}  \frac{  \log \epsilon}{2 \pi} \int_{Q}f(x)\, dx\,,
\]
where $\delta_{2,n}$ is the Kronecker delta symbol and where
\[
\mathfrak{U}:\mathopen]-\epsilon_0,\epsilon_0\mathclose[\times \mathcal{O}_{\phi_0} \times C^{1,\alpha}(\partial \Omega) \times C^0_{q,\omega,\rho}(\mathbb{R}^n)\to C^2(\overline{V})
\]
is real analytic (here $C^2(\overline{V})$ could be replaced with $C^k(\overline{V})$ for any $k>0$). The formula above makes evident that for $n=2$ the singular behavior of the solution only appears as soon as $f$ has non-zero integral over $Q$. The same is true for $n\ge 3$, as we can deduce comparing this result with those in \cite{Mu13}.   

Usually, boundary value problems in singularly perturbed domains are studied with the methods of Asymptotic Analysis, as we can find in the works of Kozlov,  Maz'ya and  Movchan \cite{KoMaMo99}, Maz'ya, Movchan, and Nieves \cite{MaMoNi13}, Mazya,  Nazarov and Plamenewskii \cite{MaNaPl00a,MaNaPl00b}, Novotny and Soko\l owski \cite{NoSo13}, and so on. In this paper we adopt a different approach, named {\it Functional Analytic Approach}, which is more suited  to obtain representation formulas in terms of analytic functions (see \cite{DaLaMu21} for a detailed introduction).  For this specific problem we will employ some periodic potential theory and we will exploit the idea of \cite{MuMi15} (later developed in \cite{LaMu16} and \cite{LuMuPu19}) of using the periodic Newtonian potential corrected with a suitable multiple of the periodic fundamental solution of the Laplace equation.  The singular behavior that arises when $f$ has non-zero integral over $Q$ is related to the fact that, in that case, the  problem
\begin{equation}\label{eq:bvpfull}
\left\{
\begin{array}{ll}
\Delta u (x)=f(x) & \forall x\in \mathbb{R}^n,\,
\\
u(x+ q z)=u(x) &\forall x \in \mathbb{R}^n,\, \forall z \in \mathbb{Z}^n,
\end{array}
\right.
\end{equation}
has no solution (as one can check applying the Divergence Theorem to $u$ in $Q$). Problem \eqref{eq:bvpfull}  may be seen as the limit of \eqref{bvpe} as $\epsilon$ goes to $0$. Indeed, in a recent paper Feppon and Ammari \cite{FeAm22} have related the appearance of singular behaviors to the failure of similar compatibility conditions. The problems studied in \cite{FeAm22} are not too far from the one that we analyze in this paper, but involve some specific generalized periodicity conditions. We also mention that  the result that we present here is, in a sense, a periodic counterpart of a result obtained  by Lanza de Cristoforis \cite{La07, La08} for a bounded domain with a single small hole.

%

The paper is organized as follows: In Section \ref{s:prel} we list some preliminary results of periodic potential theory that are used in Section \ref{s:intform} to  transform problem \eqref{bvpe} into an equivalent system of integral equations. In Section \ref{s:rep} we prove our main Theorem \ref{thm:repsol}.

\section{Preliminaries of potential theory}\label{s:prel}

As mentioned in the introduction, we use periodic potential theory to transform problem \eqref{bvpe} into an equivalent system of integral equations. More precisely, we use the periodic double layer potential, whose definition differs from that of the classical double layer potential because we replace the fundamental solution of the Laplace 
operator $\Delta = \sum_{j=1}^n\partial^{2}_{x_j}$ with a periodic analog. This will be a $q$-periodic tempered distribution $S_{q,n}$ such that
	\begin{equation*}
	\Delta S_{q,n}=\sum_{z\in \mathbb{Z}^n}\delta_{qz}-\frac{1}{|Q|_n}\,,
	\end{equation*}
where $\delta_{qz}$ denotes the Dirac distribution with mass in $qz$ and where $|\cdot|_n$ denotes the $n$-dimensional measure of a set. To define $S_{q,n}$ we can take
	$$S_{q,n}(x):=-\sum_{ z\in \mathbb{Z}^n\setminus\{0\} }\frac{1}{|Q|_n4\pi^{2}|q^{-1}z|^{2}   }e^{2\pi i (q^{-1}z)\cdot x}\,,$$
where the series converges in the sense of distributions on $\mathbb{R}^n$ (cf., e.g., Ammari and Kang~\cite[p.~53]{AmKa07}, \cite[\S 2.1]{DaLaMu21}). It can be shown that $S_{q,n}$ is real analytic in $\mathbb{R}^n\setminus q\mathbb{Z}^n$ and is locally integrable in $\mathbb{R}^n$ (cf., e.g., \cite[Thms.~12.3, 12.4]{DaLaMu21}). 

We will also find useful to write  $S_{q,n}$ as the sum of the classical fundamental solution of the Laplacian 
\[
S_n(x) := \begin{cases}
\frac{1}{s_2}\log|x| &\forall x \in \mathbb{R}^2 \setminus \{0\}, \ \ \mbox{if} \ n=2 ,\\
\frac{1}{(2-n)s_n}|x|^{2-n} &\forall x \in \mathbb{R}^n \setminus \{0\}, \ \ \mbox{if} \ n \geq 3,\\
\end{cases}
\]
and  a remainder $R_{q,n}:=S_{q,n}-S_n$ which is regular around the origin. Here above $s_n$ is the $(n-1)$-dimensional measure of the boundary $\partial \mathbb{B}_n(0,1)$ of the unit ball $\mathbb{B}_n(0,1)$ in $\mathbb{R}^n$. We note that   $R_{q,n}$  has an analytic extension to $(\mathbb{R}^n\setminus q\mathbb{Z}^n)\cup\{0\}$, which we still denote  by $R_{q,n}$, and that
\[
\Delta R_{q,n}=\sum_{z\in\mathbb{Z}^n\setminus\{0\}}\delta_{qz}-\frac{1}{|Q|_n}
\]
in the sense of distributions (see,  e.g., \cite[Thm.~12.4]{DaLaMu21}).

We now  recall the definition of  the classical (not periodic) double layer potential. We introduce another set $\tilde\Omega$, which we use as a dummy for our definitions: $\tilde{\Omega}$ is a bounded open connected subset of $\mathbb{R}^n$ of class $C^{1,\alpha}$ such that $\mathbb{R}^n \setminus \overline{\tilde{\Omega}}$  is connected. 
The classical double layer potential supported on $\tilde{\Omega}$ and with density $\theta\in C^{1,\alpha}(\partial\tilde{\Omega})$ is defined by
\[
w_{\tilde{\Omega}}[\theta](t):=
-\int_{\partial\tilde{\Omega}}\nu_{\tilde{\Omega}}(s)\cdot DS_n(t-s)\theta(s)\,d\sigma_{s}
\qquad\forall t\in \mathbb{R}^n\,,
\]
where $\nu_{\tilde{\Omega}}$ denotes the outward unit normal to $\partial{\tilde{\Omega}}$  and the symbol ``$\cdot$'' denotes the scalar product in $\mathbb{R}^n$.   As is well known,  the restriction $w_{\tilde{\Omega}}[\theta]_{|\tilde{\Omega}}$ extends to a function $w_{\tilde{\Omega}}^+[\theta]$  in  $C^{1,\alpha}(\overline{\tilde{\Omega}})$ and  the restriction $w_{\tilde{\Omega}}[\theta]_{|\mathbb{R}^n\setminus\overline{\tilde{\Omega}}}$ extends to a function $w^-_{\tilde \Omega}[\theta]$  in  $C^{1,\alpha}_{\mathrm{loc}}(\mathbb{R}^n\setminus{\tilde{\Omega}})$. Moreover, at  the boundary we have the jump formula: 
\[
w_{\tilde{\Omega}}^{\pm}[\theta]_{|\partial{\tilde{\Omega}}}={\pm}\frac{1}{2}\theta+w_{\tilde{\Omega}}[\theta]_{|\partial \tilde{\Omega}} \qquad \forall\theta\in C^{1,\alpha}(\partial{\tilde{\Omega}})\,
\]
(cf.~Folland \cite[Ch.~3]{Fo95}, \cite[\S~4.5]{DaLaMu21}).

To define the periodic double layer potential we need some more notation about periodic domains. If  $\tilde\Omega_Q$ 
is an arbitrary subset of $\mathbb{R}^n$  such that
${\overline{\tilde\Omega_Q}}\subseteq Q$ (another dummy set), we set
\[
{\mathbb{S}} [\tilde\Omega_Q]:= 
\bigcup_{z\in\mathbb{Z}^n }(qz+\tilde\Omega_Q)=q\mathbb{Z}^n+\tilde\Omega_Q\,,
\qquad
{\mathbb{S}} [\tilde\Omega_Q]^{-}:= \mathbb{R}^n\setminus{\overline{\mathbb{S}[\tilde\Omega_Q]}}\,.
\]
Then a function $u$  from ${\overline{{\mathbb{S}}[\tilde\Omega_Q]}}$ or from ${\overline{{\mathbb{S}}[\tilde\Omega_Q]^{-}}}$ to ${\mathbb{R}}$ 
is $q$-periodic if  $u(x+qz)=u(x)$ for all $x$ in the domain of definition of $u$ and for all  $z \in \mathbb{Z}^n$.
For
$j\in \{0,1\}$ and $\alpha\in\mathopen]0,1[$, we denote by $C^{j,\alpha}_{q}({\overline{{\mathbb{S}}[\tilde\Omega_Q]}} )$ and $C^{j,\alpha}_{q}({\overline{{\mathbb{S}}[\tilde\Omega_Q]^{-}}} )$ the spaces of  $q$-periodic functions of  class $C^{j,\alpha}$ in ${\overline{{\mathbb{S}}[\tilde\Omega_Q]}}$ and in ${\overline{{\mathbb{S}}[\tilde\Omega_Q]^-}}$, respectively (cf.~\cite[p.~491]{DaLaMu21}).  

If $\tilde\Omega_Q$ is of class $C^{1,\alpha}$, then  the periodic double layer potential with density $\mu\in C^{1,\alpha}(\partial\tilde\Omega_Q)$ is defined by 
\[
w_{q,\tilde\Omega_Q}[\mu](x):=
-\int_{\partial\tilde\Omega_Q}\nu_{\tilde\Omega_Q}(y)\cdot DS_{q,n}(x-y)\mu(y)\,d\sigma_{y}
\qquad\forall x\in {\mathbb{R}}^{n}\,,
\]
and we see that the expression in the right-hand side differs from that in the  definition of  $w_{\tilde\Omega}[\mu]$ because we replace $S_n$ with $S_{q,n}$.

It is well known that   the restriction $w_{q,\tilde\Omega_Q}[\mu]_{|\mathbb{S}_q[\tilde\Omega_Q]}$ extends to a function $w_{q,\tilde\Omega_Q}^+[\mu]$  of  $C^{1,\alpha}_q(\overline{\mathbb{S}_q[\tilde\Omega_Q]})$ and  the restriction $w_{q,\tilde\Omega_Q}[\mu]_{|\mathbb{S}_q[\tilde\Omega_Q]^-}$ extends to a function $w_{q,\tilde\Omega_Q}^-[\mu]$  of  $C^{1,\alpha}_{q}(\overline{\mathbb{S}_q[\tilde\Omega_Q]^-})$. Moreover,  on the boundary of $\tilde\Omega_Q$ we have the jump formula
\[
w_{q,\tilde\Omega_Q}^{\pm}[\mu]_{|\partial{{\tilde\Omega_Q}}}={\pm}\frac{1}{2}\mu+w_{q,\tilde\Omega_Q}[\mu]_{|\partial \tilde\Omega_Q} \qquad \forall\mu\in C^{1,\alpha}(\partial{{\tilde\Omega_Q}})\,
\]
(cf., e.g., \cite[Thm.~12.10]{DaLaMu21}). 

\medskip

As mentioned in the introduction, we will use  Roumieu analytic functions. The advantage is that the composition operator 
\[
(u,v)\mapsto u\circ v
\]
is real analytic in the pair of $(u,v)$ if $u$  is taken in a Roumieu class and $v$ in a Schauder space (see Proposition \ref{prop:Pr} below). Also, Roumieu analytic functions produce Roumieu analytic Newtonian potentials (see Theorem \ref{thm:newperpot}). 
So, for all bounded open subsets $\tilde\Omega$ of $\mathbb{R}^n$ and $\rho>0$, we set
\[
C_{\omega,\rho}^0(\overline{\tilde\Omega}) := \left\{u \in C^\infty(\overline{\tilde\Omega})\colon \sup_{\beta \in \mathbb{N}^n}\frac{\rho^{|\beta|}}{|\beta|!}\|D^\beta u\|_{C^0(\overline{\tilde\Omega})}<+\infty \right\}\,,
\]
and 
\[
\|u\|_{C_{\omega,\rho}^0(\overline{\tilde\Omega})}:=  \sup_{\beta \in \mathbb{N}^n}\frac{\rho^{|\beta|}}{|\beta|!}\|D^\beta u\|_{C^0(\overline{\tilde\Omega})} \qquad \forall u \in C_{\omega,\rho}^0(\overline{\tilde\Omega})\,,
\]
where $|\beta|:= \beta_1+\dots+\beta_n$ is the length of the multi-index $\beta:=(\beta_1,\dots,\beta_n)\in \mathbb{N}^n$. As is well known, the Roumieu class $\bigl(C_{\omega,\rho}^0(\overline{\tilde\Omega}),\|\cdot\|_{C_{\omega,\rho}^0(\overline{\tilde\Omega})}\bigr)$ is a Banach space. \par


If $k \in \mathbb{N}$, then we set
\[
C^{k}_{q}(\mathbb{R}^n):=
\{
u\in C^{k}(\mathbb{R}^n):\,
u(x+qz)=u(x) \quad \forall x \in \mathbb{R}^n\, , \forall z\in \mathbb{Z}^n
\}\,,
\]
and
\[
C^{\infty}_{q}(\mathbb{R}^n):=
\{
u\in C^{\infty}(\mathbb{R}^n):\,
u(x+qz)=u(x) \quad \forall x \in \mathbb{R}^n\, , \forall z\in \mathbb{Z}^n
\}\,.
\]
Similarly, if $\rho>0$, we set
\begin{equation}
\label{proum}
C_{q,\omega,\rho}^0(\mathbb{R}^n) := \left\{u \in C^\infty_q(\mathbb{R}^n)\colon \sup_{\beta \in \mathbb{N}^n}\frac{\rho^{|\beta|}}{|\beta|!}\|D^\beta u\|_{C^0(\overline{Q})}<+\infty \right\}\,,
\end{equation}
and 
\[
\|u\|_{C_{q,\omega,\rho}^0(\mathbb{R}^n)}:=  \sup_{\beta \in \mathbb{N}^n}\frac{\rho^{|\beta|}}{|\beta|!}\|D^\beta u\|_{C^0(\overline{Q})} \qquad \forall u \in C_{q,\omega,\rho}^0(\mathbb{R}^n)\,.
\]
We can see that the periodic Roumieu class $\bigl(C_{q,\omega,\rho}^0(\mathbb{R}^n),\|\cdot\|_{C_{q,\omega,\rho}^0(\mathbb{R}^n)}\bigr)$ is a Banach space. 

\medskip

It is common to use Newtonian potentials to convert boundary value problems for the Poisson equation into boundary value problems for the Laplace equation. To keep this tradition alive  we need to introduce a periodic analog of the Newtonian potential: if  $h\in C^{0}_{q}(\mathbb{R}^n)$, then we set 
\[
P_{q}[h](x):= \int_{Q}S_{q,n}(x-y)h(y)\,dy\qquad\forall x\in {\mathbb{R}}^{n}\,.
\]
Some of the properties of the periodic Newtonian potential are listed in the following theorem (we refer to  \cite{DaLaMu20} for an exhaustive overview). 

\begin{theorem}\label{thm:newperpot}
The following statements hold.
\begin{enumerate}
\item[(i)] Let $f \in C^{1}_{q}(\mathbb{R}^n)$. Then $P_{q}[f] \in C^{2}_{q}(\mathbb{R}^n)$ and
\[
\Delta P_{q}[f](x)=f(x)-\frac{1}{|Q|_n}\int_{Q}f(y)\,dy \qquad \forall x \in \mathbb{R}^n\,.
\]
\item[(ii)] Let $\rho>0$. Then there exists $\rho' \in \mathopen]0,\rho]$ such that $P_q[f] \in C^0_{q,\omega,\rho'}(\mathbb{R}^n)$ for all $f \in C^0_{q,\omega,\rho}(\mathbb{R}^n)$ and such that $P_q[\cdot]$ is linear and continuous from $C^0_{q,\omega,\rho}(\mathbb{R}^n)$ to $C^0_{q,\omega,\rho'}(\mathbb{R}^n)$.
\end{enumerate}
\end{theorem}

Then we introduce a slight variant of   Preciso \cite[Prop.~4.2.16, p.~51]{Pr98} and
\cite[Prop.~1.1, p.~101]{Pr00} on the real analyticity of a composition operator (see also Lanza de Cristoforis \cite[Prop.~2.17, Rem.~2.19]{La98} and \cite[Prop.~9, p.~214]{La07}).

\begin{proposition}\label{prop:Pr}
Let $m$, $h$, $k \in \mathbb{N}$, $h$, $k \geq 1$. Let $\alpha \in \mathopen]0,1]$, $\rho>0$. Let $\Omega'$, $\Omega''$ be bounded open connected  subsets of $\mathbb{R}^h$, $\mathbb{R}^k$, respectively. Let $\Omega''$ be of class $C^1$. Then the operator $T$ defined by
\[
T[u,v]:= u\circ v
\]
for all $(u,v)\in C^0_{\Omega',\rho}(\overline{\Omega'}) \times C^{m,\alpha}(\overline{\Omega''},\Omega')$ is real analytic from the open subset $C^0_{\Omega', \rho}(\overline{\Omega'})\times C^{m,\alpha}(\overline{\Omega''}, \Omega')$ of $C^0_{\Omega',\rho}(\overline{\Omega'})\times C^{m,\alpha}(\overline{\Omega''}, \mathbb{R}^
h)$ to $C^{m,\alpha}(\overline{\Omega''})$.
\end{proposition}

Finally, we need a last technical  lemma about the real analytic dependence of certain maps related to the change of variables in integrals and to the pullback of the outer normal field.
For a proof we refer to Lanza de Cristoforis and Rossi \cite[p.~166]{LaRo04}
and to Lanza de Cristoforis \cite[Prop. 1]{La07}.
\begin{lemma}\label{rajacon}
 Let $\alpha$, $\Omega$ be as in \eqref{Omega_def}.  Then the following statements hold.
\begin{itemize}
\item[(i)] For each $\psi \in \mathcal{A}^{1,\alpha}_{\partial \Omega}$, there exists a unique  
$\tilde \sigma[\psi] \in C^{0,\alpha}(\partial\Omega)$ such that $\tilde \sigma[\psi] > 0$ and 
\[ 
  \int_{\psi(\partial\Omega)}w(s)\,d\sigma_s=  \int_{\partial\Omega}w \circ \psi(y)\tilde\sigma[\psi](y)\,d\sigma_y, \qquad \forall w \in L^1(\psi(\partial\Omega)).
\]
Moreover, the map $\tilde \sigma[\cdot]$ from $\mathcal{A}^{1,\alpha}_{\partial \Omega}$ to $ C^{0,\alpha}(\partial\Omega)$ is real analytic.
\item[(ii)] The map from $\mathcal{A}^{1,\alpha}_{\partial \Omega}$ to $ C^{0,\alpha}(\partial\Omega, \mathbb{R}^{n})$ that takes $\psi$ to $\nu_{\mathbb{I}[\psi]} \circ \psi$ is real analytic.
\end{itemize}
\end{lemma}

 \section{Formulation of problem (\ref{bvpe})  in terms of integral equations}\label{s:intform}
 
First we convert problem \eqref{bvpe} into a Dirichlet problem for the Laplace equation. In order to do so, we would use a function  whose Laplacian equals the right-hand side $f \in C^0_{q,\omega,\rho}(\mathbb{R}^n)$ of the first equation in problem \eqref{bvpe}. The natural candidate would be the periodic Newtonian potential $P_{q}[  f]$. However, we have
\begin{equation}\label{eq:DeltaP}
\Delta P_{q}[  f] (x)=  f(x)-\frac{1}{|Q|_n}\int_{Q}f(y)\, dy \qquad \forall x\in 
 \mathbb{S}[\Omega_{\epsilon,\phi}]^-\, .
\end{equation}
So we need to get rid of the term $\frac{1}{|Q|_n}\int_{Q}f(y)\, dy$  in equation \eqref{eq:DeltaP}: We note  that the function from $\mathbb{R}^n \setminus (p+q\mathbb{Z}^n)$ to $\mathbb{R}$ that takes $x$ to $-  S_{q,n}(x-p)\int_{Q}f(y)\, dy$ is $q$-periodic and analytic, and that
 \[
 \Delta \bigg [-  S_{q,n}(x-p)\int_{Q}f(y)\, dy\bigg ]=\frac{1}{|Q|_n}\int_{Q}f(y)\, dy \qquad \forall x \in \mathbb{R}^n \setminus (p+q\mathbb{Z}^n)\, .
 \]
 As a consequence,
\begin{equation}\label{eq:DeltaP-S}
 \Delta \bigg [  P_{q}[f](x)-  S_{q,n}(x-p)\int_{Q}f(y)\, dy\bigg ]=  f(x) \qquad \forall x \in  \mathbb{S}[\Omega_{\epsilon,\phi}]^-\,
\end{equation}
 and we can use the corrected Newtonian potential 
 \[
 P_{q}[f](x)-  S_{q,n}(x-p)\int_{Q}f(y)\, dy
 \]
to transform problem \eqref{bvpe} into a Dirichlet problem for the Laplace equation,  that in turn we can analyze using the periodic double layer potential.

 This is what we do to prove the following Theorem \ref{biju}: first we transform the Dirichlet-Poisson problem into a Dirichlet-Laplace problem, and then we represent the solution as the sum of a  constant and a double layer potential with a  density that satisfies a certain boundary integral equation pulled-back to $\partial \Omega$.
 
 \begin{theorem}\label{biju}
 Let $\alpha \in \mathopen]0,1[$. Let $\rho>0$. Let $p\in Q$. Let $\Omega$ be as in \eqref{Omega_def}. Let $(\phi_0,g_0,f_0) \in \mathcal{A}^{1,\alpha}_{\partial \Omega} \times C^{1,\alpha}(\partial \Omega) \times C^0_{q,\omega,\rho}(\mathbb{R}^n)$. Let $\epsilon_{0}$, $\mathcal{O}_{\phi_0}$ be as in \eqref{e0}.  Let $(\epsilon,\phi,g,f) \in \mathopen]0,\epsilon_{0}\mathclose[\times \mathcal{O}_{\phi_0} \times C^{1,\alpha}(\partial \Omega) \times C^0_{q,\omega,\rho}(\mathbb{R}^n)$. Then problem \eqref{bvpe} has a unique solution $u[\epsilon,\phi,g,f]$ in $C^{1,\alpha}_{q}(\overline{\mathbb{S}[\Omega_{\epsilon,\phi}]^-})$, which is delivered by the formula
\begin{equation}\label{biju1}
\begin{split}
u[\epsilon,\phi,g,f](x):=\,\, &\omega(\epsilon,\phi,g,f,x)+ P_{q}[f]( x )-  S_{q,n}(x-p)\int_{Q}f(y)\, dy\\&+\delta_{2,n} \frac{  \log \epsilon}{2 \pi} \int_{Q}f(y)\, dy\qquad\qquad\qquad
\forall x\in \overline{\mathbb{S}[\Omega_{\epsilon,\phi}]^-}\,,
\end{split}
\end{equation}
where
\[
\omega(\epsilon,\phi,g,f,x):=
w_{q, \Omega_{\epsilon,\phi}}^{-}[ \theta\circ \phi^{(-1)}(\epsilon^{-1}(\cdot- p))]( x )+c
\qquad\forall x\in \overline{\mathbb{S}[\Omega_{\epsilon,\phi}]^-}
\]
and where $(\theta,c)$ is the unique solution in $C^{1,\alpha}(\partial\Omega) \times{\mathbb{R}}$ of  the system of integral equations
\begin{eqnarray}
\label{biju2}
\lefteqn{
-\frac{1}{2}\theta (t) 
-   \int_{\partial\Omega}\nu_{\mathbb{I}[\phi]}\circ \phi(s)\cdot DS_{n}(  \phi(t)-\phi(s))\theta (s)\tilde{\sigma}[\phi](s)  \,d\sigma_{s}
}
\\ \nonumber
&&\qquad\qquad\quad\quad
-\epsilon^{n-1}\int_{\partial\Omega}\nu_{\mathbb{I}[\phi]}\circ \phi(s)\cdot DR_{q,n}(\epsilon (\phi(t)-\phi(s)))\theta (s)\tilde{\sigma}[\phi](s)  \,d\sigma_{s} +c
\\ \nonumber
&&\qquad\qquad\quad
=
g(t)
-  \int_{Q}S_{ q,n}( p+\epsilon \phi(t)-y) f(y)\,dy\\ \nonumber
&&\qquad\qquad\quad
\ +\frac{1}{\epsilon^{n-2}} S_{n}( \phi(t)) \int_{Q}f(y)\, dy+  R_{q,n}( \epsilon \phi(t)) \int_{Q}f(y)\, dy
\qquad\forall t\in\partial \Omega\,,\\
\label{biju2bis}
\lefteqn{
\int_{\partial \Omega}\theta \tilde{\sigma}[\phi]\, d\sigma=0\, .}
\end{eqnarray}
\end{theorem}
\begin{proof}
By \eqref{eq:DeltaP-S} we can see that  a function $u \in 
C^{1,\alpha}_{q}(\overline{\mathbb{S}[\Omega_{\epsilon,\phi}]^-})$ solves problem \eqref{bvpe} if and only if  the function 
\[
u(x)- P_{q}[f](x)+  S_{q,n}(x-p)\int_{Q}f(y)\, dy-\delta_{2,n} \frac{  \log \epsilon}{2 \pi} \int_{Q}f(y)\, dy
\qquad  \forall x\in \overline{\mathbb{S}[\Omega_{\epsilon,\phi}]^-}\,
\]
is a solution of the following boundary value problem
\begin{equation}
\label{bvpe3}
\left\{
\begin{array}{ll}
\Delta \omega  (x)=0 & \forall x\in \mathbb{S}[\Omega_{\epsilon,\phi}]^-\,,
\\
\omega(x+qz)=\omega(x)& \forall x\in \overline{\mathbb{S}[\Omega_{\epsilon,\phi}]^-}\, ,\forall z \in \mathbb{Z}^n\, ,
\\
\omega(x)= g\circ\phi^{(-1)}(\epsilon^{-1}(x-  p))
- P_{q}[ f](x)&\\
\qquad\quad+  S_{q,n}(x-p)\int_{Q}f(y)\, dy-\delta_{2,n} \frac{  \log \epsilon}{2 \pi} \int_{Q}f(y)\, dy
& \forall x\in  \partial\Omega_{\epsilon,\phi}\,.
\end{array}
\right.
\end{equation}
By \cite[Prop.~12.24]{DaLaMu21} the solution of problem \eqref{bvpe3}
exists, is unique, belongs to $C^{1,\alpha}_{q}(\overline{\mathbb{S}[\Omega_{\epsilon,\phi}]^-})$, and can be written as 
\[
\omega(x)=w_{q, \Omega_{\epsilon,\phi}}^{-}[ \theta\circ \phi^{(-1)}(\epsilon^{-1}(\cdot- p))](x)+c\quad
\forall x\in \overline{\mathbb{S}[\Omega_{\epsilon,\phi}]^-}\,,
\]
where  $(\theta,c)$ is the unique pair in $C^{1,\alpha}(\partial\Omega)\times {\mathbb{R}}$ such that  
  $\int_{\partial \Omega}\theta \tilde{\sigma}[\phi]\, d\sigma=0$ and
\begin{eqnarray*}
\lefteqn{
-\frac{1}{2}\theta\circ \phi^{(-1)} (\epsilon^{-1}(x- p))+
w_{q,\Omega_{\epsilon,\phi}}[\theta\circ \phi^{(-1)}(\epsilon^{-1}(\cdot- p))](x)+c
}
\\ \nonumber
&&\qquad\qquad\qquad 
=
g\circ \phi^{(-1)}(\epsilon^{-1}(x- p))
- P_{q}[f](x)\\ \nonumber
&&\qquad\qquad\qquad\quad
+  S_{q,n}(x-p)\int_{Q}f(y)\, dy-\delta_{2,n} \frac{  \log \epsilon}{2 \pi} \int_{Q}f(y)\, dy\qquad\forall x\in\partial \Omega_{\epsilon,\phi}
 \,.
\end{eqnarray*}
By a change of variable, the last equation can be rewritten as 
\begin{equation}\label{eqabove}
\begin{split}
-\frac{1}{2}&\theta (t) 
-\epsilon^{n-1}  \int_{\partial\Omega}\nu_{\Omega_{\epsilon,\phi}}
( p +\epsilon \phi(s))
 DS_{ q,n}( p +\epsilon \phi(t) -
( p+\epsilon \phi(s)))\theta (s) \tilde{\sigma}[\phi](s) \,d\sigma_{s} +c
\\ 
&\qquad\quad
=
g(t)
- \int_{Q}S_{ q,n}( p +\epsilon \phi(t)- y) f(y)\,dy \\ 
&\qquad\quad\quad+  S_{q,n}(p+\epsilon \phi(t)-p)\int_{Q}f(y)\, dy-\delta_{2,n} \frac{  \log \epsilon}{2 \pi} \int_{Q}f(y)\, dy
\qquad\forall t\in\partial \Omega\,. 
\end{split}
\end{equation}
Then we note that
\[
\begin{split}
S_{q,n}(\epsilon \phi(t))\int_{Q}f(y)\, dy=&\,\frac{1}{\epsilon^{n-2}}S_{n}(\phi(t))\int_{Q}f(y)\, dy+\delta_{2,n}\frac{\log \epsilon}{2\pi} \int_{Q}f(y)\, dy \\&+ R_{q,n}(\epsilon \phi(t))\int_{Q}f(y)\, dy \qquad \forall t \in \partial \Omega
\end{split}
\]
and, since
\[
\nu_{\Omega_{\epsilon,\phi}}
( p +\epsilon \phi(s))=\nu_{\mathbb{I}[\phi]}\circ \phi(s) \qquad \forall s \in \partial \Omega
\]
(cf., e.g., Lanza de Cristoforis \cite[Lem.~3.1]{La08}),
equation \eqref{eqabove} can be rewritten as
\[
\begin{split}
-\frac{1}{2}&\theta (t) 
-\epsilon^{n-1} \int_{\partial\Omega}\nu_{\mathbb{I}[\phi]}\circ \phi(s)
DS_{ q,n}(\epsilon (\phi(t)-\phi(s)))\theta (s)\tilde{\sigma}[\phi](s)  \,d\sigma_{s} +c
\\ 
&\qquad\quad
=
g(t)
-  \int_{Q}S_{ q,n}( p+\epsilon \phi(t)-y) f(y)\,dy\\ 
&\qquad\quad
\ +\frac{1}{\epsilon^{n-2}} S_{n}( \phi(t)) \int_{Q}f(y)\, dy+  R_{q,n}( \epsilon \phi(t)) \int_{Q}f(y)\, dy
\qquad\forall t\in\partial \Omega\,,
\end{split}
\]
which is easily seen to be equivalent to \eqref{biju2}.
\end{proof}
 
 We denote by $(\theta_{\epsilon,\phi,g,f},c_{\epsilon,\phi,g,f})$ the unique solution of \eqref{biju2}--\eqref{biju2bis}. We would like, however, to have integral equations that are defined also for $\epsilon=0$, and system \eqref{biju2}-\eqref{biju2bis} is not. So we  rescale and take 
 \[
(\theta^\#_{\epsilon,\phi,g,f},c^\#_{\epsilon,\phi,g,f}):=\epsilon^{n-2} (\theta_{\epsilon,\phi,g,f},c_{\epsilon,\phi,g,f})
 \]
 for all $(\epsilon,\phi,g,f) \in \mathopen]0,\epsilon_{0}\mathclose[\times \mathcal{O}_{\phi_0} \times C^{1,\alpha}(\partial \Omega) \times C^0_{q,\omega,\rho}(\mathbb{R}^n)$. We see that $(\theta^\#_{\epsilon,\phi,g,f},c^\#_{\epsilon,\phi,g,f})$ coincide with the unique pair  $(\theta^\#,c^\#)\in C^{1,\alpha}(\partial\Omega) \times{\mathbb{R}}$ such that
\begin{eqnarray}
\label{biju2sharp}
\lefteqn{
-\frac{1}{2}\theta^\# (t) 
-   \int_{\partial\Omega}\nu_{\mathbb{I}[\phi]}\circ \phi(s)\cdot DS_{n}(  \phi(t)-\phi(s))\theta^\# (s) \tilde{\sigma}[\phi](s) \,d\sigma_{s}
}
\\ \nonumber
&&\qquad\quad\quad
-\epsilon^{n-1}\int_{\partial\Omega}\nu_{\mathbb{I}[\phi]}\circ \phi(s)\cdot DR_{q,n}(\epsilon (\phi(t)-\phi(s)))\theta^\# (s) \tilde{\sigma}[\phi](s) \,d\sigma_{s} +c^\#
\\ \nonumber
&&\qquad\quad
=
\epsilon^{n-2} g(t)
-  \epsilon^{n-2} \int_{Q}S_{ q,n}( p+\epsilon \phi(t)-y) f(y)\,dy\\ \nonumber
&&\qquad\quad
\ +S_{n}( \phi(t)) \int_{Q}f(y)\, dy+  \epsilon^{n-2} R_{q,n}( \epsilon \phi(t)) \int_{Q}f(y)\, dy
\qquad\forall t\in\partial \Omega\,,\\
\label{biju2sharpbis}\lefteqn{
\int_{\partial \Omega}\theta^\# \tilde{\sigma}[\phi]\, d\sigma=0\,,}
\end{eqnarray}
and \eqref{biju2sharp}-\eqref{biju2sharpbis} makes sense also for $\epsilon=0$.

Using $(\theta^\#_{\epsilon,\phi,g,f},c^\#_{\epsilon,\phi,g,f})$ instead of $(\theta_{\epsilon,\phi,g,f},c_{\epsilon,\phi,g,f})$ we obtain from  Theorem \ref{biju} the following alternative representation formula for $u[\epsilon,\phi,g,f]$.

 \begin{corollary}\label{bijusharp}
 Let $\alpha \in \mathopen]0,1[$. Let $\rho>0$.  Let $p\in Q$. Let $\Omega$ be as in \eqref{Omega_def}.  Let $(\phi_0,g_0,f_0) \in \mathcal{A}^{1,\alpha}_{\partial \Omega} \times C^{1,\alpha}(\partial \Omega) \times C^0_{q,\omega,\rho}(\mathbb{R}^n)$. Let $\epsilon_{0}$, $\mathcal{O}_{\phi_0}$ be as in \eqref{e0}.  Let $(\epsilon,\phi,g,f) \in \mathopen]0,\epsilon_{0}\mathclose[\times \mathcal{O}_{\phi_0} \times C^{1,\alpha}(\partial \Omega) \times C^0_{q,\omega,\rho}(\mathbb{R}^n) $. Then the  unique solution $u[\epsilon,\phi,g,f]$ 
 in $C^{1,\alpha}_{q}(\overline{\mathbb{S}[\Omega_{\epsilon,\phi}]^-})$ of problem \eqref{bvpe} (which is also given 
 by \eqref{biju1}) can be written as
\[
\begin{split}
u[\epsilon,\phi,g,f](x)  = \,\, &\frac{1}{\epsilon^{n-2}}w_{q, \Omega_{\epsilon,\phi}}^{-}[ \theta^\#_{\epsilon,\phi,g,f}\circ \phi^{(-1)}(\epsilon^{-1}(\cdot- p))]( x )+\frac{1}{\epsilon^{n-2}}c^\#_{\epsilon,\phi,g,f}\\
&+ P_{q}[f]( x )-  S_{q,n}(x-p)\int_{Q}f(y)\, dy\\&+\delta_{2,n} \frac{  \log \epsilon}{2 \pi} \int_{Q}f(y)\, dy\qquad\qquad\qquad
\forall x\in \overline{\mathbb{S}[\Omega_{\epsilon,\phi}]^-}\,,
\end{split}
\]
where $(\theta^\#_{\epsilon,\phi,g,f},c^\#_{\epsilon,\phi,g,f})$ is the unique solution in  $C^{1,\alpha}(\partial\Omega) \times{\mathbb{R}}$  of system \eqref{biju2sharp}-\eqref{biju2sharpbis}.
\end{corollary}

For $(\epsilon,\phi,g,f)$ that tends to  $(0,\phi_0,g_0,f_0)$ system \eqref{biju2sharp}-\eqref{biju2sharpbis} turns into the following ``limiting system of integral equations,''
\begin{eqnarray}
\label{biju2sharplim}
\lefteqn{
-\frac{1}{2}\theta^\# (t) 
-   \int_{\partial\Omega}\nu_{\mathbb{I}[\phi_0]}\circ \phi_0(s)\cdot DS_{n}(  \phi_0(t)-\phi_0(s))\theta^\# (s) \tilde{\sigma}[\phi_0](s) \,d\sigma_{s} + c^\#
}
\\  \nonumber
&&\qquad\qquad\qquad\qquad\quad
=
\delta_{2,n} g_0(t)
-  \delta_{2,n} \int_{Q}S_{ q,n}( p-y) f_0(y)\,dy\\ \nonumber
&&\qquad\qquad\qquad\qquad\quad
\ +S_{n}( \phi_0(t)) \int_{Q}f_0(y)\, dy+  \delta_{2,n} R_{q,n}( 0) \int_{Q}f_0(y)\, dy
\qquad\forall t\in\partial \Omega\,,
\\
\label{biju2sharplimbis}\lefteqn{
\int_{\partial \Omega}\theta^\# \tilde{\sigma}[\phi_0]\, d\sigma=0\, .}
\end{eqnarray}

In the following Theorem \ref{thm:limiting} we prove that \eqref{biju2sharplim}-\eqref{biju2sharplimbis} has a solution and that such solution is unique. As we shall see, system \eqref{biju2sharplim}-\eqref{biju2sharplimbis} is related to a specific boundary value problem, which we call the  ``limiting boundary value problem.'' The proof of Theorem \ref{thm:limiting} follows the guidelines of the proof of  \cite[Lem.~3.4]{Mu12}.
\begin{theorem}\label{thm:limiting}
 Let $\alpha \in \mathopen]0,1[$. Let $\rho>0$.  Let $p\in Q$. Let $\Omega$ be as in \eqref{Omega_def}.  Let $(\phi_0,g_0,f_0) \in \mathcal{A}^{1,\alpha}_{\partial \Omega} \times C^{1,\alpha}(\partial \Omega) \times C^0_{q,\omega,\rho}(\mathbb{R}^n)$.  Let $\tilde{\tau}$ be the unique solution in $C^{0,\alpha}(\partial\Omega)$ of 
\begin{equation}\label{limadj}
\left\{
\begin{array}{ll}
-\frac{1}{2}\tau(t)+\int_{\partial\Omega}\nu_{\mathbb{I}[\phi_0]}\circ \phi_0(t)\cdot DS_{n}(\phi_0(t)-\phi_0(s))\tau(s)\tilde{\sigma}[\phi_0](s)\,d\sigma_{s}=0&\forall t\in\partial\Omega\,,
\\
\int_{\partial\Omega}\tau \tilde{\sigma}[\phi_0] d\sigma=1  \,.  &
\end{array}
\right. 
\end{equation}
Then the following statements hold. 
\begin{enumerate}
\item[(i)] The limiting system \eqref{biju2sharplim}-\eqref{biju2sharplimbis} has one and only one solution $(\tilde{\theta}^\#,\tilde{c}^\#)$ in $C^{1,\alpha}(\partial\Omega) \times {\mathbb{R}}$. Moreover, 
\[
\tilde{c}^\#=\int_{\partial\Omega}g^\#_n\tilde{\tau}\tilde{\sigma}[\phi_0]\,d\sigma\,,
\]
where, for all $t\in \partial\Omega$ we have
\[
g^\#_n(t):=
\left\{
\begin{array}{ll}
 g_0(t)-  \int_{Q}S_{ q,n}( p-y) f_0(y)\,dy+S_{n}( \phi_0(t)) \int_{Q}f_0(y)\,dy & \\
 \qquad +   R_{q,n}( 0) \int_{Q}f_0(y)\,dy& \text{if $n=2$}\, ,\\
S_{n}( \phi_0(t)) \int_{Q}f_0(y)\,dy& \text{if $n\geq 3$}\, .
\end{array}
\right.
\]
\item[(ii)] The limiting boundary value problem
\begin{equation}
\label{bvplim}
\left\{
\begin{array}{ll}
\Delta u (x)=0 & \forall x\in {\mathbb{R}}^{n}\setminus\overline{\mathbb{I}[\phi_0]}\,,
\\
u(x)=g^\#_n\circ \phi_0^{(-1)}( x)  &  \forall x\in \phi_0(\partial\Omega)\,,
\\
\lim_{x\to\infty}u (x)=\tilde{c}^\# \,,&
\end{array}
\right.
\end{equation}
has one and only one solution $\tilde{u}^\# $ in $ C^{1,\alpha}_{
{\mathrm{loc}} }({\mathbb{R}}^{n}\setminus \mathbb{I}[\phi_0])$. Moreover, 
\begin{equation}
\label{thm:limiting2}
\tilde{u}^\#(x)=w^-_{\mathbb{I}[\phi_0]}[\tilde{\theta}^\#\circ \phi_0^{(-1)}](x)+\tilde{c}^\#\qquad\forall x\in 
{\mathbb{R}}^{n}\setminus \mathbb{I}[\phi_0]\,.
\end{equation}
\end{enumerate}
\end{theorem}
\begin{proof}
By classical potential theory (cf., e.g., Folland~\cite[Ch.~3 ]{Fo95}, \cite[Lem.~6.34]{DaLaMu21}) and by the theorem of change of variable in integrals, we can see that problem \eqref{limadj} has a unique solution $\tilde{\tau}\in C^{0,\alpha}(\partial\Omega)$. Then the validity of statement (i) follows by \cite[Thm.~6.37 (i)]{DaLaMu21} and, once more, by the theorem of change of variable in integrals.\par

To prove statement (ii) we first observe that problem 
 \eqref{bvplim}  has at most one continuous solution (this follows by a classical argument based on the Maximum Principle).  Then, by the properties of the classical double layer potentials and exploiting the fact that $(\tilde{\theta}^\#,\tilde{c}^\#)$ is a solution of \eqref{biju2sharplim}-\eqref{biju2sharplimbis}, we can see that the function $\tilde{u}^\#$  in \eqref{thm:limiting2} is harmonic and satisfies   the second and the third conditions  in 
\eqref{bvplim} (cf.~\cite[\S~4.5]{DaLaMu21}). Thus it coincides with the unique solution of  \eqref{bvplim}. 
 \end{proof}

We might be curious to know what is the constant $\tilde c^\#$ that appears in Theorem \ref{thm:limiting}. In the following remark we attempt an explanation.  

\begin{remark}\label{remarko}
By the computations of \cite[Lem.~7.2]{DaMuRo15}, we can verify that  if $n=2$, then $\tilde{u}^\#$ coincides with the unique solution of
\[
\left\{
\begin{array}{ll}
\Delta u (x)=0 & \forall x\in {\mathbb{R}}^{2}\setminus\overline{\mathbb{I}[\phi_0]}\,,
\\
u(x)=g^\#_n\circ \phi_0^{(-1)}( x)  &  \forall x\in \phi_0(\partial\Omega)\,,
\\
\sup_{x\in \mathbb{R}^2\setminus \mathbb{I}[\phi_0]}|u (x)|<+\infty\,,&
\end{array}
\right.
\]
and that 
\[
\tilde{c}^\#=\lim_{x\to \infty}\tilde{u}^\#(x)\, .
\]
If, instead,  $n\geq 3$, then an extension of the $2$-dimensional argument of \cite[Lem.~7.2]{DaMuRo15} to the $(n\ge 3)$-dimensional case shows that
\[
\tilde{c}^\#=\frac{1}{(2-n)s_n}\int_{Q}f_0(y)\, dy \, \Big(\lim_{x\to \infty}|x|^{n-2}H_0(x)\Big)^{-1}\, ,
\]
where $H_0$ is the unique function of $C^{1,\alpha}_\mathrm{loc}(\mathbb{R}^n\setminus \mathbb{I}[\phi_0])$ such that 
\[
\label{bvplimvarbis}
\left\{
\begin{array}{ll}
\Delta H_0 (x)=0 & \forall x\in {\mathbb{R}}^{n}\setminus\overline{\mathbb{I}[\phi_0]}\,,
\\
H_0(x)=1  &  \forall x\in \phi_0(\partial\Omega)\,,
\\
\lim_{x\to \infty}H_0(x)=0&
\end{array}
\right.
\]
and where we can see that the limit 
\[
\lim_{x\to \infty}|x|^{n-2}H_0(x)
\]
exists and belongs to $\mathopen]0,+\infty[$ (see, e.g., Folland \cite[Chap.~2, Prop.~2.74]{Fo95}). In particular,  we have
\[
\text{$\tilde{c}^\#\neq 0$ as soon as $\int_{Q}f_0(y)dy\neq 0$.}
\] 
\end{remark}

In the following Theorem \ref{thm:Lmbdcase1} we consider the map that takes $(\epsilon,\phi,g,f)$ to $(\theta^\#_{\epsilon,\phi,g,f},c^\#_{\epsilon,\phi,g,f})$ and prove that it has a real analytic continuation in a neighborhood of $(0,\phi_0,g_0,f_0)$. The proof of Theorem \ref{thm:Lmbdcase1} exploits the Implicit Function Theorem for real analytic maps in Banach spaces (cf., e.g.,  Deimling \cite[Thm.~15.3]{De85}).

\begin{theorem}\label{thm:Lmbdcase1}
 Let $\alpha \in \mathopen]0,1[$. Let $\rho>0$.  Let $p\in Q$. Let $\Omega$ be as in \eqref{Omega_def}.  Let $(\phi_0,g_0,f_0) \in \mathcal{A}^{1,\alpha}_{\partial \Omega} \times C^{1,\alpha}(\partial \Omega) \times C^0_{q,\omega,\rho}(\mathbb{R}^n)$. Let $\epsilon_{0}$, $\mathcal{O}_{\phi_0}$ be as in \eqref{e0}.   Then there exist $\epsilon_{\#,1} \in \mathopen ]0,\epsilon_0[$, an open neighborhood $\mathcal{O}_{\phi_0}'$ of $\phi_0$ in $\mathcal{A}^{1,\alpha}_{\partial \Omega}$,  an open neighborhood $\mathcal{U}_0$ of $(g_0,f_0)$ in $C^{1,\alpha}(\partial \Omega) \times C^0_{q,\omega,\rho}(\mathbb{R}^n)$, and a real analytic map 
 \[
 (\Theta_\#,C_\#):\mathopen]-\epsilon_{\#,1},\epsilon_{\#,1}\mathclose[ \times \mathcal{O}_{\phi_0}' \times \mathcal{U}_0\to C^{1,\alpha}(\partial \Omega)\times \mathbb{R}
 \] 
 such that
\begin{align}
&(\Theta_\#[\epsilon,\phi,g,f],C_\#[\epsilon,\phi,g,f])=(\theta^\#_{\epsilon,\phi,g,f},c^\#_{\epsilon,\phi,g,f}) \qquad \forall (\epsilon,\phi,g,f)  \in  \mathopen]0,\epsilon_{\#,1}\mathclose[ \times \mathcal{O}_{\phi_0}' \times \mathcal{U}_0\, , \nonumber
\\ 
&(\Theta_\#[0,\phi_0,g_0,f_0],C_\#[0,\phi_0,g_0,f_0])=(\tilde{\theta}^\#,\tilde{c}^\#)\, .\nonumber
\end{align}
\end{theorem}
\begin{proof}
 Let $\Lambda_{\#}:= (\Lambda_{\#,1}, \Lambda_{\#,2})$ be the map 
from $ \mathopen]-\epsilon_{0},\epsilon_{0}\mathclose[\times \mathcal{O}_{\phi_0} \times C^{1,\alpha}(\partial \Omega) \times C^0_{q,\omega,\rho}(\mathbb{R}^n)\times C^{1,\alpha}(\partial\Omega) \times {\mathbb{R}}$ to $C^{1,\alpha}(\partial\Omega)\times \mathbb{R}$ defined by
\begin{eqnarray}
\nonumber 
\lefteqn{
\Lambda_{\#,1}[\epsilon,\phi,g,f,\theta^\#,c^\#](t):=-\frac{1}{2}\theta^\# (t) 
-   \int_{\partial\Omega}\nu_{\mathbb{I}[\phi]}\circ \phi(s)\cdot DS_{n}(  \phi(t)-\phi(s))\theta^\# (s) \tilde{\sigma}[\phi](s) \,d\sigma_{s}
}
\\ \nonumber
&&\qquad\qquad\quad\quad
-\epsilon^{n-1}\int_{\partial\Omega}\nu_{\mathbb{I}[\phi]}\circ \phi(s)\cdot DR_{q,n}(\epsilon (\phi(t)-\phi(s)))\theta^\# (s) \tilde{\sigma}[\phi](s) \,d\sigma_{s} +c^\#
\\ \nonumber
&&\qquad\qquad\quad
-\epsilon^{n-2} g(t)
+  \epsilon^{n-2} \int_{Q}S_{ q,n}( p+\epsilon \phi(t)-y) f(y)\,dy\\ \nonumber
&&\qquad\qquad\quad
\ -S_{n}( \phi(t)) \int_{Q}f(y)\, dy-  \epsilon^{n-2} R_{q,n}( \epsilon \phi(t)) \int_{Q}f(y)\, dy
\qquad\forall t\in\partial \Omega\,,\\
\nonumber 
\lefteqn{
\Lambda_{\#,2}[\epsilon,\phi,g,f,\theta^\#,c^\#]:=\int_{\partial \Omega}\theta^\# \tilde{\sigma}[\phi]\, d\sigma\, ,}
\end{eqnarray}
for all  $(\epsilon,\phi,g,f,\theta^\#,c^\#)\in  \mathopen]-\epsilon_{0},\epsilon_{0}\mathclose[\times \mathcal{O}_{\phi_0} \times C^{1,\alpha}(\partial \Omega) \times C^0_{q,\omega,\rho}(\mathbb{R}^n)\times C^{1,\alpha}(\partial\Omega) \times {\mathbb{R}}$.

 We first note that equation $\Lambda_\#[0,\phi_0,g_0,f_0,\theta^\#,c^\#]=0$  in the unknown $(\theta^\#,c^\#)\in C^{1,\alpha}(\partial\Omega) \times {\mathbb{R}}$ is equivalent  to the limiting  system \eqref{biju2sharplim}-\eqref{biju2sharplimbis}, which has one and only one solution $(\tilde{\theta}^\#,\tilde{c}^\#)$ in $C^{1,\alpha}(\partial\Omega) \times {\mathbb{R}}$. Similarly, if $(\epsilon,\phi,g,f)\in   \mathopen]0,\epsilon_{0}\mathclose[\times \mathcal{O}_{\phi_0} \times C^{1,\alpha}(\partial \Omega) \times C^0_{q,\omega,\rho}(\mathbb{R}^n)$, then equation $\Lambda_\#[\epsilon,\phi,g,f,\theta^\#,c^\#]=0$ in the unknown $(\theta^\#,c^\#)\in C^{1,\alpha}(\partial\Omega) \times {\mathbb{R}}$ is equivalent  to the system \eqref{biju2sharp}-\eqref{biju2sharpbis} and has one and only one solution $(\theta^\#_{\epsilon,\phi,g,f},c^\#_{\epsilon,\phi,g,f})\in C^{1,\alpha}(\partial\Omega) \times {\mathbb{R}}$.
 
  Then we observe that $\Lambda_\#$ is real analytic in a neighborhood of $(0, \phi_0,g_0,f_0,\tilde{\theta}^\#,\tilde{c}^\#)$. Namely, it is real analytic from $\mathopen]-\epsilon_{0},\epsilon_{0}\mathclose[\times \mathcal{O}_{\phi_0} \times C^{1,\alpha}(\partial \Omega) \times C^0_{q,\omega,\rho}(\mathbb{R}^n)\times C^{1,\alpha}(\partial\Omega) \times {\mathbb{R}}$ to $C^{1,\alpha}(\partial\Omega)\times \mathbb{R}$. This follows by the real analyticity results for the double layer potential of Lanza de Cristoforis and Rossi \cite[Thm.~4.11 (iii)]{LaRo08}, by real analyticity results for integral operators with real analytic kernel (cf.~\cite{LaMu13}), by the regularity result for volume potentials of Theorem \ref{thm:newperpot}, by the analyticity results for the composition operator of  Valent~\cite[Thm.~5.2, p.~44]{Va88}, by Proposition \ref{prop:Pr}, by Lemma \ref{rajacon},  and by standard calculus in Banach spaces.

Since we plan to use the Implicit Function Theorem, we now consider the partial differential $\partial_{(\theta^\#,c^\#)}\Lambda_\#[0,\phi_0,g_0,f_0, \tilde{\theta}^\#,\tilde{c}^\#]$ of $\Lambda_\#$ at $(0,\phi_0,g_0,f_0, \tilde{\theta}^\#,\tilde{c}^\#)$ with respect to the variable $(\theta^\#,c^\#)$.  By standard calculus in Banach spaces we have
\begin{align*}
&\partial_{(\theta^\#,c^\#)}\Lambda_{\#,1}[0,\phi_0,g_0,f_0,\tilde{\theta}^\#,\tilde{c}^\#](\overline{\theta},\overline{c})(t)\\&=-\frac{1}{2}\overline{\theta} (t) 
-   \int_{\partial\Omega}\nu_{\mathbb{I}[\phi_0]}\circ \phi_0(s)\cdot DS_{n}(  \phi_0(t)-\phi_0(s))\overline{\theta} (s)\tilde{\sigma}[\phi_0](s)  \,d\sigma_{s}+\overline{c}
\qquad\forall t\in\partial \Omega\,,\\
& \partial_{(\theta^\#,c^\#)}\Lambda_{\#,2}[0,\phi_0,g_0,f_0,\tilde{\theta}^\#,\tilde{c}^\#](\overline{\theta},\overline{c})=\int_{\partial \Omega}\overline{\theta} \tilde{\sigma}[\phi_0]\, d\sigma\, ,
\end{align*}
for all $(\overline{\theta},\overline{c})\in C^{1,\alpha}(\partial\Omega) \times {\mathbb{R}}$. By arguing as in the proofs of Theorem \ref{thm:limiting} (i) and of \cite[Prop.~13.10]{DaLaMu21}, we can see that $\partial_{(\theta^\#,c^\#)}\Lambda_\#[0,\phi_0,g_0,f_0, \tilde{\theta}^\#,\tilde{c}^\#]$ is a bijection.
Then by the Open Mapping Theorem, the operator  $\partial_{(\theta^\#,c^\#)}\Lambda_\#[0,\phi_0,g_0,f_0, \tilde{\theta}^\#,\tilde{c}^\#]$ is also a homeomorphism from $C^{1,\alpha}(\partial\Omega) \times {\mathbb{R}}$ to itself.

We can invoke the Implicit Function Theorem for real analytic maps in Banach spaces (cf., e.g.,  Deimling \cite[Thm.~15.3]{De85}) and deduce  the existence of $\epsilon_{\#,1}$, $\mathcal{O}_{\phi_0}'$,  $\mathcal{U}_0$, and $(\Theta_\#,C_\#)$ as in the statement. 
\end{proof}

\section{A functional analytic representation theorem for the solution}\label{s:rep}

We are ready to prove our main Theorem \ref{thm:repsol}. As mentioned in the introduction, we will write the map $(\epsilon,\phi,g,f)\mapsto u[\epsilon,\phi,g,f]_{\overline{V}}$ as a combination of real analytic maps of $(\epsilon,\phi,g,f)$ and--possibly singular but completely known--elementary functions of $\epsilon$. In particular, we will focus on the case where $(\epsilon,\phi,g,f)$ is close to a quadruple $(0,\phi_0,g_0,f_0)$ with the size parameter $\epsilon$ equal to $0$, which is interesting also because of the singular behavior that appears when $\int_Q f_0\,dx\neq 0$. Theorem \ref{thm:repsol} is a consequence of Theorem \ref{thm:Lmbdcase1} on the analytic continuation of   $(\epsilon,\phi,g,f)\mapsto(\theta^\#_{\epsilon,\phi,g,f},c^\#_{\epsilon,\phi,g,f})$ and of the representation formula  for $u[\epsilon,\phi,g,f]$ of Corollary \ref{bijusharp}.

\begin{theorem}\label{thm:repsol}
 Let $\alpha \in \mathopen]0,1[$. Let $\rho>0$.  Let $p\in Q$. Let $\Omega$ be as in \eqref{Omega_def}.  Let $(\phi_0,g_0,f_0) \in \mathcal{A}^{1,\alpha}_{\partial \Omega} \times C^{1,\alpha}(\partial \Omega) \times C^0_{q,\omega,\rho}(\mathbb{R}^n)$. Let $\epsilon_{\#,1}$, $\mathcal{O}_{\phi_0}'$,  $\mathcal{U}_0$ be  as in Theorem \ref{thm:Lmbdcase1}. Let $V$ be a bounded open subset of $\mathbb{R}^n \setminus (p+q\mathbb{Z}^n)$. Then there exist $\epsilon_{\#,2} \in \mathopen ]0,\epsilon_{\#,1}[$, an open neighborhood $\mathcal{O}_{\phi_0}''$ of $\phi_0$ in $\mathcal{A}^{1,\alpha}_{\partial \Omega}$ contained in $\mathcal{O}_{\phi_0}'$,  and a real analytic map $\mathfrak{U}$ from $\mathopen]-\epsilon_{\#,2},\epsilon_{\#,2}\mathclose[ \times \mathcal{O}_{\phi_0}'' \times \mathcal{U}_0$ to $C^{2}(\overline{V})$ such that
 \begin{equation}\label{eq:repsol:1}
 \overline{V}\subseteq \mathbb{S}[\Omega_{\epsilon,\phi}]^- \qquad \forall (\epsilon,\phi)\in \mathopen]-\epsilon_{\#,2},\epsilon_{\#,2}\mathclose[ \times \mathcal{O}_{\phi_0}'' 
 \end{equation}
 and 
 \begin{equation}\label{eq:repsol:2}
 \begin{split}
u[\epsilon,\phi,g,f]_{|\overline{V}}=\frac{1}{\epsilon^{n-2}}\mathfrak{U}[\epsilon,\phi,g,f]&+\delta_{2,n}  \frac{  \log \epsilon}{2 \pi} \int_{Q}f(y)\, dy\\
 &  \forall (\epsilon,\phi,g,f)\in \mathopen]0,\epsilon_{\#,2}\mathclose[ \times \mathcal{O}_{\phi_0}'' \times \mathcal{U}_0\, .
\end{split}
\end{equation}
Moreover,
 \begin{equation}\label{eq:repsol:3}
\mathfrak{U}[0,\phi_0,g_0,f_0](x)=\tilde{c}^\#+\delta_{2,n} P_{q}[f_0]( x )-  \delta_{2,n} S_{q,n}(x-p)\int_{Q}f_0(y)\,dy \qquad \forall x \in \overline{V}\, .
\end{equation}
\end{theorem}
\begin{proof} Clearly, \eqref{eq:repsol:1} holds true for  $\epsilon_{\#,2} $ and $\mathcal{O}_{\phi_0}''$  small enough. Then we note that
\[
\begin{split}
u[\epsilon,\phi,g,f](x)=\,\, &-  \epsilon \int_{\partial\Omega}\nu_{\mathbb{I}[\phi]}\circ \phi(s)\cdot DS_{q,n}(  x-p-\epsilon \phi(s))\Theta_\#[\epsilon,\phi,g,f] (s) \tilde{\sigma}[\phi](s) \,d\sigma_{s}\\
&+\frac{1}{\epsilon^{n-2}}C_\#[\epsilon,\phi,g,f]+ P_{q}[f]( x )-  S_{q,n}(x-p)\int_{Q}f(y)\, dy\\&+\delta_{2,n} \frac{  \log \epsilon}{2 \pi} \int_{Q}f(y)\, dy
\end{split}
\]
for all $x\in \overline{V}$ and all $(\epsilon,\phi,g,f)\in \mathopen]0,\epsilon_{\#,2}\mathclose[ \times \mathcal{O}_{\phi_0}'' \times \mathcal{U}_0$ and we set
\[
\begin{split}
\mathfrak{U}&[\epsilon,\phi,g,f](x):=- \epsilon^{n-1} \int_{\partial\Omega}\nu_{\mathbb{I}[\phi]}\circ \phi(s)\cdot DS_{q,n}(  x-p-\epsilon \phi(s))\Theta_\#[\epsilon,\phi,g,f] (s) \tilde{\sigma}[\phi](s) \,d\sigma_{s}\\
&+C_\#[\epsilon,\phi,g,f]+ \epsilon^{n-2} P_{q}[f]( x )-  \epsilon^{n-2} S_{q,n}(x-p)\int_{Q}f(y)\, dy
\end{split}
\]
for all $x\in \overline{V}$ and all  $(\epsilon,\phi,g,f)\in \mathopen]-\epsilon_{\#,2},\epsilon_{\#,2}\mathclose[ \times \mathcal{O}_{\phi_0}'' \times \mathcal{U}_0$. By  Theorem \ref{thm:newperpot}, by Lemma \ref{rajacon}, by real analyticity results for integral operators with real analytic kernel (cf.~\cite{LaMu13}), and by standard calculus in Banach spaces, we deduce that $\mathfrak{U}$ is a real analytic map  from $\mathopen]-\epsilon_{\#,2},\epsilon_{\#,2}\mathclose[ \times \mathcal{O}_{\phi_0}'' \times \mathcal{U}_0$ to $C^{2}(\overline{V})$ such that equality \eqref{eq:repsol:2} holds. Since $C_\#[0,\phi_0,g_0,f_0]=\tilde{c}^\#$, we also deduce the validity of \eqref{eq:repsol:3}.
\end{proof}

If we  fix $(\phi,g,f) =(\phi_0,g_0,f_0) \in \mathcal{A}^{1,\alpha}_{\partial \Omega} \times C^{1,\alpha}(\partial \Omega) \times C^0_{q,\omega,\rho}(\mathbb{R}^n)$ with 
\[
\int_{Q}f_0\, dx\neq 0\, ,
\]
then formula \eqref{eq:repsol:2} shows that for $n=2$  the function $u[\epsilon,\phi_0,g_0,f_0]$ displays a singular behavior of order $\log\epsilon$ as $\epsilon$ tends to $0$. Under the same assumptions on the triple 
$(\phi,g,f)$, we can see that, for $n\ge 3$,  $u[\epsilon,\phi_0,g_0,f_0]$ has a singularity of order $\epsilon^{2-n}$   as $\epsilon$ tends to $0$. This can be deduced from  \eqref{eq:repsol:2} and \eqref{eq:repsol:3} and remembering that $\tilde c^\#\neq 0$ (cf.~Remark \ref{remarko}).

Also, the fact that $\mathfrak{U}$ is real analytic means that we can expand $\mathfrak{U}[\epsilon,\phi,g,f]$ into a power series  that converges (in norm) for $(\epsilon,\phi,g,f)$ in a neighborhood of $(0,\phi_0,g_0,f_0)$. The approach presented in this paper can be used to compute the corresponding coefficients (see, e.g., \cite{DaLaMu21, DaMuPu19, DaMuRo15}).

\subsection*{Acknowledgment}

The authors are members of the ``Gruppo Nazionale per l'Analisi Matematica, la Probabilit\`a e le loro Applicazioni'' (GNAMPA) of the ``Istituto Nazionale di Alta Matematica'' (INdAM). 
P.M.~acknowledges the support from EU through the H2020-MSCA-RISE-2020 project EffectFact, 
Grant agreement ID: 101008140.

\end{document}